\documentclass[12pt, reqno]{amsart}%
\usepackage{amsmath, amsthm, amscd, amsfonts, amssymb, graphicx, color}
\usepackage[bookmarksnumbered, colorlinks, plainpages]{hyperref}%
\usepackage{amsmath}%
\usepackage{amsfonts}%
\usepackage{amssymb}%
\usepackage{graphicx}
\providecommand{\U}[1]{\protect \rule{.1in}{.1in}}
\newtheorem{theorem}{Theorem}[section]

\newtheorem{proposition}[theorem]{Proposition}
\newtheorem{corollary}[theorem]{Corollary}
\theoremstyle{definition}
\newtheorem{definition}[theorem]{Definition}

\theoremstyle{remark}
\newtheorem{remark}[theorem]{Remark}
\numberwithin{equation}{section}

\begin{document}
\title[Spinor Frenet equations in three dimensional Lie groups]{Spinor Frenet equations in three dimensional Lie groups}

\begin{abstract}
In this paper, we study spinor Frenet equations in three dimensional Lie
Groups with a bi-invariant metric. Also, we obtain spinor Frenet equations for
special cases of three dimensional Lie groups.

\end{abstract}
\author{O. Zeki Okuyucu}
\address{Department of Mathematics, Bilecik \c{S}eyh Edebali University, Turkey}
\email{osman.okuyucu@bilecik.edu.tr}
\author{\"{O}. G\"{o}kmen Y\i ld\i z}
\address{Department of Mathematics, Bilecik \c{S}eyh Edebali University, Turkey}
\email{ogokmen.yildiz@bilecik.edu.tr}
\author{Murat Tosun}
\address{Department of Mathematics, Sakarya University, Turkey}
\email{tosun@sakarya.edu.tr}
\maketitle

\setcounter{page}{1}

\setcounter{page}{1}

%\dedicatory{This paper is dedicated to Professor ABCD}

\setcounter{page}{1}

%\dedicatory{This paper is dedicated to Professor ABCD}

\section{Introduction}

In differential geometry, we think that curves are geometric set of points of
loci. Curves theory is a important workframe in the differential geometry
studies. Geometric properties of a curve are heavily studied for a long time
and are still studied in Euclidean space and other spaces. One of the most
important tools used to analyze a curve is the Frenet frame, a moving frame
that provides an orthonormal coordinate system at each point of the curve. We
can show that this orthonormal system with $\left \{  T,N,B\right \}  $ in
Euclidean $3$-space, which are called tangent vector field, principal normal
vector field and binormal vector field, respectively. And for a curve, we can
calculate curvature and torsion along the curve with the help of Frenet
vectors. They are so useful for characterizations of special curves like that
general helices, Slant helices, Mannheim curves, Bertrand curves etc.
Characterizations of these curves are studied in different spaces and we can
see the applications of special curves in nature, mechanic tools, computer
aided design and computer graphics etc.

Recently, Castillo and Barrales studied spinor formulation of the Frenet
equations of the curve and they obtained spinor equivalent of the Frenet
equations for a curve \cite{castillo}. Moreover Ki\c{s}i et.al. and \"{U}nal
et. al. defined spinor formulations of the Frenet equations of a curve
according to Darboux Frame and Bishop Frame, respectively (see \cite{kisi} and
\cite{unal}).

In this paper, we define spinor formulation of the Frenet equations of the
curve in three dimensional Lie Groups. Also we give this formulation for
special cases of three dimensional Lie groups. We believe that this study will
be useful for curves theory in Lie groups.

\section{Preliminaries}

Let $G$ be a Lie group with a bi-invariant metric $\left \langle \text{
},\right \rangle $ and $D$ be the Levi-Civita connection of Lie group $G.$ If
$\mathfrak{g}$ denotes the Lie algebra of $G$ then we know that $\mathfrak{g}
$ is issomorphic to $T_{e}G$ where $e$ is neutral element of $G.$ If
$\left \langle \text{ },\right \rangle $ is a bi-invariant metric on $G$ then we
have%
\begin{equation}
\left \langle X,\left[  Y,Z\right]  \right \rangle =\left \langle \left[
X,Y\right]  ,Z\right \rangle \label{2-1}%
\end{equation}
and
\begin{equation}
D_{X}Y=\frac{1}{2}\left[  X,Y\right] \label{2-2}%
\end{equation}
for all $X,Y$ and $Z\in \mathfrak{g}.$

%\dedicatory{This paper is dedicated to Professor ABCD}

Let $\alpha:I\subset \mathbb{R\rightarrow}G$ be an arc-lenghted curve and
$\left \{  X_{1},X_{2,}...,X_{n}\right \}  $ be an orthonormal basis of
$\mathfrak{g}.$ In this case, we write that any two vector fields $W$ and $Z$
along the curve $\alpha \ $as $W=\sum_{i=1}^{n}w_{i}X_{i}$ and $Z=\sum
_{i=1}^{n}z_{i}X_{i}$ where $w_{i}:I\rightarrow \mathbb{R}$ and $z_{i}%
:I\rightarrow \mathbb{R}$ are smooth functions. Also the Lie bracket of two
vector fields $W$ and $Z$ is given
\[
\left[  W,Z\right]  =\sum_{i=1}^{n}w_{i}z_{i}\left[  X_{i},X_{j}\right]
\]
and the covariant derivative of $W$ along the curve $\alpha$ with the notation
$D_{\alpha^{\shortmid}}W$ is given as follows%
\begin{equation}
D_{\alpha^{\shortmid}}W=\overset{\cdot}{W}+\frac{1}{2}\left[  T,W\right]
\label{2-3}%
\end{equation}
where $T=\alpha^{\prime}$ and $\overset{\cdot}{W}=\sum_{i=1}^{n}\overset
{\cdot}{w_{i}}X_{i}$ or $\overset{\cdot}{W}=\sum_{i=1}^{n}\frac{dw}{dt}X_{i}.$
Note that if $W$ is the left-invariant vector field to the curve $\alpha$ then
$\overset{\cdot}{W}=0$ (see \cite{crouch} for details).

%\dedicatory{This paper is dedicated to Professor ABCD}

Let $G$ be a three dimensional Lie group and $\left(  T,N,B,\varkappa
,\tau \right)  $ denote the Frenet apparatus of the curve $\alpha$, and
calculate $\varkappa=\overset{\cdot}{\left \Vert T\right \Vert }.$

%\dedicatory{This paper is dedicated to Professor ABCD}

\begin{definition}
\label{2.2}Let $\alpha:I\subset \mathbb{R\rightarrow}G$ be a parametrized curve
with the Frenet apparatus $\left(  T,N,B,\varkappa,\tau \right)  $ then
\begin{equation}
\tau_{G}=\frac{1}{2}\left \langle \left[  T,N\right]  ,B\right \rangle
\label{2-4}%
\end{equation}
or
\[
\tau_{G}=\frac{1}{2\varkappa^{2}\tau}\overset{\cdot \cdot \text{
\  \  \  \  \  \  \  \ }\cdot}{\left \langle T,\left[  T,T\right]  \right \rangle
}+\frac{1}{4\varkappa^{2}\tau}\overset{\text{ \  \ }\cdot}{\left \Vert \left[
T,T\right]  \right \Vert ^{2}}%
\]
(see \cite{ciftci}).
\end{definition}

%\dedicatory{This paper is dedicated to Professor ABCD}

\begin{proposition}
\label{pro 2.3}Let $\alpha:I\subset \mathbb{R\rightarrow}G$ be an arc length
parametrized curve with the Frenet apparatus $\left \{  T,N,B\right \}  $. Then
the following equalities%
\begin{align*}
\left[  T,N\right]   &  =\left \langle \left[  T,N\right]  ,B\right \rangle
B=2\tau_{G}B\\
\left[  T,B\right]   &  =\left \langle \left[  T,B\right]  ,N\right \rangle
N=-2\tau_{G}N
\end{align*}
hold (see \cite{zeki}).
\end{proposition}

%\dedicatory{This paper is dedicated to Professor ABCD}

In the following sentences we give some properties about spinors:

We can represent a spinor
\[
\phi=\left(
\begin{array}
[c]{c}%
\phi_{1}\\
\phi_{2}%
\end{array}
\right)
\]
by way of three vector \ $a,b,c\in%
%TCIMACRO{\U{211d} }%
%BeginExpansion
\mathbb{R}
%EndExpansion
^{3}$ so that%
\begin{equation}
a+ib=\phi^{t}\sigma \phi,\text{ \  \ }c=-\widehat{\phi^{t}}\sigma \phi
,\label{2-5}%
\end{equation}
where the superscript $t$\ denotes transposition and $\widehat{\phi}$ is the
mate ( or conjugate \cite{cartan})
\[
\widehat{\phi}\equiv-\left(
\begin{array}
[c]{cc}%
0 & 1\\
-1 & 0
\end{array}
\right)  \text{, }\overline{\phi}=\left(
\begin{array}
[c]{c}%
-\overline{\phi}_{2}\\
\overline{\phi}_{1}%
\end{array}
\right)
\]
where the bar indicates complex conjugation. In addition, $\sigma=\left(
\sigma_{1},\sigma_{2},\sigma_{3}\right)  $\ is a vector whose cartesian
components are the complex symmetrics $2\times2$\ matrices%
\begin{equation}
\sigma_{1}=\left(
\begin{array}
[c]{cc}%
1 & 0\\
0 & -1
\end{array}
\right)  ,\text{ }\sigma_{2}=\left(
\begin{array}
[c]{cc}%
i & 0\\
0 & i
\end{array}
\right)  ,\text{ }\sigma_{3}=\left(
\begin{array}
[c]{cc}%
0 & -1\\
-1 & 0
\end{array}
\right)  .\label{2-6}%
\end{equation}
Then, the vector $a,b$ and $c$ are clearly given by%
\begin{align*}
a+ib  & =\left(  \phi_{1}^{2}-\phi_{2}^{2},i(\phi_{1}^{2}+\phi_{2}^{2}%
),-2\phi_{1}\phi_{2}\right) \\
c  & =\left(  \phi_{1}\overline{\phi}_{2}+\overline{\phi}_{1}\phi_{2},i\left(
\overline{\phi}_{2}\phi_{1}-\overline{\phi}_{1}\phi_{2}\right)  ,\left \vert
\phi_{1}\right \vert ^{2}-\left \vert \phi_{2}\right \vert ^{2}\right)
\end{align*}
Since the vector $a+ib$ is a isotropic vector, by way of an explict
computation one finds that $a,b$ and $c$ are one another orthogonal and
$\left \vert a\right \vert =\left \vert b\right \vert =\left \vert c\right \vert
=\overline{\phi}^{t}\phi$. Moreover $\left \langle a\times b,c\right \rangle
=\det(a,b,c)>0.$

Conversely, if the vector $a,b$ and $c$ one another orthogonal vectors of same
magnitude $\left(  \left \langle a\times b,c\right \rangle >0\right)  ,$ then
there is a spinor defined up to sign such that the Eq. \eqref{2-5} holds.

As it is mentioned the above, for two arbitrary $\psi$ and $\phi$, there exist
following equalities
\begin{align*}
\overline{\psi^{t}\sigma \phi}  & =-\widehat{\psi^{t}}\sigma \widehat{\phi}\\
\widehat{a\psi+b\phi}  & =\overline{a}\widehat{\psi}+\overline{b}\widehat
{\phi}\\
\widehat{\widehat{\phi}}  & =-\phi
\end{align*}
where $a$\ and $b$\ are complex numbers \cite{cartan} and \cite{castillo}.

The relations between spinors and orthonormal bases given in Eq. \eqref{2-5}
is two to on; the spinors $\phi$ and $-\phi$ correspond to the same ordered
orthonormal bases $\left \{  a,b,c\right \}  $, with $\left \vert a\right \vert
=\left \vert b\right \vert =\left \vert c\right \vert $ and $\left \langle a\times
b,c\right \rangle >0.$ Moreover we can define different spinors by using the
ordered triads $\left \{  a,b,c\right \}  ,\left \{  b,c,a\right \}  $ and
$\left \{  c,a,b\right \}  .$

The equation $\psi^{t}\sigma \phi=\phi^{t}\sigma \psi$ is satisfied for any pair
of $\psi$ and $\phi$, because the matrices $\sigma$ given by Eq. \eqref{2-6}
are symmetric. For any spinor $\phi$ and its conjugate $\widehat{\phi}$ are
linearly independent, where $\phi \neq0.$

\section{Spinor Frenet equations in a three dimensional Lie groups}

In this section we define the spinor Frenet equations for a curve in a three
dimensional Lie group $G$ with a bi-invariant metric $\left \langle \text{
},\right \rangle $. Also we give the spinor Frenet equations in the special
cases of $G$.

%\dedicatory{This paper is dedicated to Professor ABCD}

Let $\alpha:I\subset \mathbb{R\rightarrow}G$ be a curve with arc-length
parameter $s$. Then we can wirte with the help of the Eq. \eqref{2-3} and
Proposition \ref{pro 2.3} the following equations for Frenet vectors of the
curve $\alpha$%
\begin{equation}%
\begin{bmatrix}
\frac{dT}{ds}\\
\frac{dN}{ds}\\
\frac{dB}{ds}%
\end{bmatrix}
=%
\begin{bmatrix}
0 & \varkappa & 0\\
-\varkappa & 0 & \left(  \tau-\tau_{G}\right) \\
0 & -\left(  \tau-\tau_{G}\right)  & 0
\end{bmatrix}%
\begin{bmatrix}
T\\
N\\
B
\end{bmatrix}
\label{3-1}%
\end{equation}
where $\left \{  T,N,B\right \}  $ is Frenet frame, $\tau_{G}=\frac{1}%
{2}\left \langle \left[  T,N\right]  ,B\right \rangle $ and $\varkappa,$ $\tau$
are curvatures of the curve $\alpha$ in three dimensional Lie group $G,$ respectively.

According to the results presented in section $2$, there exists a spinor
$\phi,$ such that%
\begin{equation}
N+iB=\phi^{t}\sigma \phi,\text{ \  \  \  \  \  \  \  \  \ }T=\widehat{\phi^{t}}%
\sigma \phi \label{3-2}%
\end{equation}
with $\overline{\phi^{t}}\phi=1$. Therefore the spinor $\phi$ denotes the
triad $\left \{  N,B,T\right \}  .$ And the Frenet equations for each point of
the curve $\alpha$ must correspond to some expression for $\frac{d\phi}{ds}.$

As $\left \{  \phi,\widehat{\phi}\right \}  $ is a basis for the two component
spinors $\left(  \phi \neq0\right)  $, there are two functions $g$ and $h$,
such that%
\begin{equation}
\frac{d\phi}{ds}=g\phi+h\widehat{\phi},\label{3-3}%
\end{equation}
where the functions $g$ and $h$ are possibly complex-valued functions.

Diferentiating the first equation in the Eq. \eqref{3-2} with respect to $s$
and by using the Frenet equations, we get%
\begin{align}
\frac{dN}{ds}+i\frac{dB}{ds}  & =\frac{d}{ds}\left(  \phi^{t}\sigma \phi \right)
\nonumber \\
-\varkappa T-i\left(  \tau-\tau_{G}\right)  \left \{  N+iB\right \}   & =\left(
\frac{d\phi}{ds}\right)  ^{t}\sigma \phi+\phi^{t}\sigma \left(  \frac{d\phi}%
{ds}\right)  .\label{3-4}%
\end{align}
If we think together the Eq. \eqref{3-3} and Eq. \eqref{3-4}, using again the
Eq. \eqref{3-2}, we get%
\[
-\varkappa T-i\left(  \tau-\tau_{G}\right)  \left \{  N+iB\right \}
=2g(N+iB)-2h(T).
\]
From the last equation we can write,%
\[
g=-i\frac{\left(  \tau-\tau_{G}\right)  }{2}\text{ and }h=\frac{\varkappa}{2}.
\]
Thus, we have proved the following theorem.

%\dedicatory{This paper is dedicated to Professor ABCD}

\begin{theorem}
\label{teo 3.1}Let $\alpha$\ be an arc-lenghted regular curve with the Frenet
vector fields $\left \{  T,N,B\right \}  $ in the Lie group $G$. If\ its
two-component spinors $\phi$ represents the triad $\left \{  N,B,T\right \}  $,
then the Frenet equations are equivalent to the single spinor equation%
\[
\frac{d\phi}{ds}=-i\frac{\left(  \tau-\tau_{G}\right)  }{2}\phi+\frac
{\varkappa}{2}\widehat{\phi},
\]
where $\tau_{G}=\frac{1}{2}\left \langle \left[  T,N\right]  ,B\right \rangle $
and $\varkappa$, $\tau$ are curvatures of the curve $\alpha.$
\end{theorem}

%\dedicatory{This paper is dedicated to Professor ABCD}

Now, we introduce this equation for special cases of three dimensional Lie
group. In the following remark, we note that three dimensional Lie groups
admitting bi-invariant metrics are $S^{3},$ $SO^{3}$ and Abelian Lie groups
using the same notation as in \cite{ciftci} and \cite{santo} as follows:

\begin{remark}
\label{rem 3.1}Let $G$ be a Lie group with a bi-invariant metric $\left \langle
,\right \rangle $. Then the following equalities can be given in
different Lie groups:
\end{remark}

\qquad \qquad \textbf{i) }If $G$ is abelian group then $\tau_{G}=0.$

\qquad \qquad \textbf{ii) }If $G$ is $SO^{3}$ then $\tau_{G}=\frac{1}{2}.$

\qquad \qquad \textbf{iii) }If $G$ is $S^{3}$ then $\tau_{G}=1$

(see for details \cite{ciftci} and \cite{santo}).

\begin{corollary}
\label{cor 3.1}Let $\alpha$\ be an arc-lenghted regular curve with the Frenet
vector fields $\left \{  T,N,B\right \}  $ in the Abelian Lie group $G$. If\ its
two-component spinors $\phi$ represents the triad $\left \{  N,B,T\right \}  $,
then the Frenet equations are equivalent to the single spinor equation%
\[
\frac{d\phi}{ds}=-i\frac{\tau}{2}\phi+\frac{\varkappa}{2}\widehat{\phi},
\]
where $\varkappa$ and $\tau$ are curvatures of the curve $\alpha.$
\end{corollary}

\begin{proof}
If $G$ is Abellian Lie group then using the above Remark and the Theorem
\ref{teo 3.1} we have the result.
\end{proof}

From the above corollary we can see easily this study is a generalization of
the spinor equivalent of the Frenet equations for a curve defined by Del
Castillo and Barrales \cite{castillo} in Euclidean 3-space. Moreover, with a
similar proof, we have the following two corollaries.

\begin{corollary}
\label{cor 3.2}Let $\alpha$\ be an arc-lenghted regular curve with the Frenet
vector fields $\left \{  T,N,B\right \}  $ in the Lie group $SO^{3}$. If\ its
two-component spinors $\phi$ represents the triad $\left \{  N,B,T\right \}  $,
then the Frenet equations are equivalent to the single spinor equation%
\[
\frac{d\phi}{ds}=-i\frac{\left(  \tau-\frac{1}{2}\right)  }{2}\phi
+\frac{\varkappa}{2}\widehat{\phi},
\]
where $\varkappa$ and $\tau$ are curvatures of the curve $\alpha.$
\end{corollary}

\begin{corollary}
\label{cor 3.3}Let $\alpha$\ be an arc-lenghted regular curve with the Frenet
vector fields $\left \{  T,N,B\right \}  $ in the Lie group $S^{3}$. If\ its
two-component spinors $\phi$ represents the triad $\left \{  N,B,T\right \}  $,
then the Frenet equations are equivalent to the single spinor equation%
\[
\frac{d\phi}{ds}=-i\frac{\left(  \tau-1\right)  }{2}\phi+\frac{\varkappa}%
{2}\widehat{\phi},
\]
where $\varkappa$ and $\tau$ are curvatures of the curve $\alpha.$
\end{corollary}

\end{document}